\documentclass[10pt]{amsart}

\usepackage{amssymb,amsthm,amsmath}
\usepackage[numbers,sort&compress]{natbib}
\usepackage{color}
\usepackage{graphicx}
\usepackage{tikz}

\title[ ]{Almost Mathieu operators  with  completely resonant phases}

\author{Wencai Liu}
\address[ Wencai Liu]{ Department of Mathematics, University of California, Irvine, California 92697-3875, USA}\email{liuwencai1226@gmail.com}

\theoremstyle{plain}
\newtheorem{theorem}{Theorem}[section]

\newtheorem{lemma}[theorem]{Lemma}

\theoremstyle{definition}
\newtheorem{definition}[theorem]{Definition}

\newtheorem{remark}[theorem]{Remark}

\begin{document}


\begin{abstract}
Let $\alpha\in \mathbb{R}\backslash \mathbb{Q}$  and $\beta(\alpha) = \limsup _{n \to \infty}(\ln q_{n+1})/ q_n <\infty$, where $p_n/q_n$ is the continued
fraction approximations  to $\alpha$. Let  $ (H_{\lambda,\alpha,\theta}u) (n)=u(n+1)+u(n-1)+  2\lambda \cos2\pi(\theta+n\alpha)u(n)$ be
the  almost Mathieu operator on $\ell^2(\mathbb{Z})$, where $\lambda, \theta\in \mathbb{R}$.
 Avila and Jitomirskaya \cite{avila2009ten} conjectured that
   for  $2\theta\in \alpha \mathbb{Z}+\mathbb{Z}$,
$H_{\lambda,\alpha,\theta}$
satisfies Anderson localization if $|\lambda|>e^{2\beta(\alpha)}$.
 In this paper, we developed a method to treat simultaneous  frequency and phase resonances  and obtain that  for  $2\theta\in \alpha \mathbb{Z}+\mathbb{Z}$,
$H_{\lambda,\alpha,\theta}$
satisfies Anderson localization if $|\lambda|>e^{3\beta(\alpha)}$.

\end{abstract}
\maketitle

\section{Introduction}
    The almost Mathieu operator (AMO) is the (discrete) quasi-periodic   Schr\"{o}dinger operator  on  $   \ell^2(\mathbb{Z})$:
 \begin{equation*}
 (H_{\lambda,\alpha,\theta}u)(n)=u({n+1})+u({n-1})+ 2\lambda \cos2\pi(\theta+n\alpha)u(n),  
 \end{equation*}
where $\lambda$ is the coupling, $\alpha $ is the frequency, and $\theta $ is the phase.
\par
The AMO is the most studied quasi-periodic Schr\"{o}dinger operator, arising naturally as a physical model.
We refer the readers to \cite{jitomirskaya2015dynamics,MR2145079}
  and the references therein  for    physical background. Most recently, there are a lot of interesting topics related to AMO, e.g. \cite{avila2016second,yang2017spectral,igor,mira,jl,jl1,shiwen}.

   We say phase $\theta\in \mathbb{R}$ is completely resonant with respect to frequency $\alpha$ if $2\theta\in \alpha \mathbb{Z}+\mathbb{Z}$. In this paper, we always assume
  $\alpha\in \mathbb{R}\backslash \mathbb{Q}$.
 \par
 {\bf Conjecture 1:} Avila and Jitomirskaya  \cite{avila2009ten,slov} assert    that  for   $2\theta\in \alpha \mathbb{Z}+\mathbb{Z}$, $H_{\lambda,\alpha,\theta}$ satisfies Anderson localization  if  $ |\lambda| >  e^{ 2 \beta}$,
where
 \begin{equation*}
  \beta= \beta(\alpha)=\limsup_{n\rightarrow\infty}\frac{\ln q_{n+1}}{q_n},
 \end{equation*}
 and  $ \frac{p_n}{q_n} $  is  the continued fraction  approximations    to $\alpha$.
 \par
Completely resonant phases   of   quasi-periodic operators  correspond to the rational  rotation numbers with respect to frequency in the Aubry dual model.  We refer the readers to \cite{gjls,jmgafa,chu89} for the Aubry duality.
The (quantitative) reducibility of Schr\"odinger cocycles with rational  rotation numbers is   related to many topics in   quasi-periodic operators.
For example, it  is  a  good  approach to  show that all the spectral gaps $G_m$ labeled by gap labeling theorem\footnote{The rotation number $\rho$ on gap $G_m$ satisfies $2\rho=m\alpha\mod \mathbb{Z}$.}  \cite{bellissard1985almost,johnson1982rotation} are open (named after dry Ten Martini Problem for the almost Mathieu operator). The dry  Ten Martini Problem \footnote{The dry Ten Martini Problem is still open for all parameters. The non-critical coupling case has been   solved by Avila-You-Zhou \cite{avila2016dry}.} is stronger  than Ten Martini Problem (the latter one was finally solved by Avila and Jitomirskaya \cite{avila2009ten}). It is also related to the H\"older continuity of Lyapunov exponents, rotation numbers and  the integrated density of states. 

The reducibility of the Schr\"odinger  cocycles with  rational  rotation  numbers  was first  established  by Moser  and P{\"o}schel \cite{moser}, who modified   the proof of reducibility of cocycles  with  Diophantine rotation numbers \cite{sinai}. See \cite{Amor,eli} for more precise  results.
It was first realized by Puig \cite{puig1,puig2} that localization at completely resonant phases leads to reducibility for Schr\"odinger cocycles with rational rotation numbers for the dual model. The argument was significantly developed in \cite{MR3338640,liu2017upper,han2017dry,avila2008almost}.

 For completely resonant phases,
 Jitomirskaya-Koslover-Schulteis  \cite{jitomirskaya2005localization}
    proved  localization  for  $\alpha\in DC $\footnote{ We say $\alpha \in \mathbb{R}\backslash \mathbb{Q} $ satisfies Diophantine condition (DC)
if there exist $\tau>1, \kappa>0$ such that
$$ ||k\alpha||\geq\kappa |k|^{-\tau}  \text{ for   any } k\in \mathbb{Z}\setminus \{0\},$$
where $||x||={\rm dist}(x,\mathbb{Z})$.} via   a simple modification of the proof in \cite{jitomirskaya1999metal}.
Their result can be extended  to $\alpha$ with $\beta(\alpha)=0$ without any difficulty. In order to avoid too many concepts, if $\beta(\alpha)=0$,  we call $\alpha$
Diophantine. To the contrary, if  $\beta(\alpha)>0$, we call $\alpha$
Liouville.

Recently,
 there have been several remarkable sharp arithmetic  transition results for all parameters. In particular,   phase transitions   happen in  positive Lyapunov exponent  regime for Liouville frequencies \cite{AYZ,jl,JK,RJ,jlcpam}.
Later, universal (reflective) hierarchical structure of eigenfunctions was established in the
localization regime \cite{jl,jl1}  with  an arithmetic condition on $\theta$.
However,  all the sharp results aforementioned excluded the completely resonant phases.  The purpose of this paper is to   consider  the missing part.

   We  prove   Conjecture 1 for $|\lambda|>e^{3\beta}$.
   That is
   \begin{theorem}\label{MaintheoremAnderson}
Suppose  frequency  $\alpha\in  \mathbb{R}\backslash \mathbb{Q}$  satisfies $\beta (\alpha)<\infty$.  Then the   almost Mathieu operator
    $H_{\lambda,\alpha,\theta}$ satisfies Anderson localization  if $2\theta\in \alpha \mathbb{Z}+\mathbb{Z}$  and  $ |\lambda| > e^{3\beta(\alpha)}$.
    Moreover, if $\phi$ is an eigenfunction, that is $H_{\lambda,\alpha,
    \theta}\phi=E\phi$, we have
    \begin{equation*}
    \limsup_{k\to \infty} \frac{\ln (\phi^2(k)+\phi^2(k-1))}{2|k|} \leq -(\ln \lambda-3\beta).
\end{equation*}
\end{theorem}
\begin{remark}

For $\alpha$ with $\beta(\alpha)=+\infty$, $H_{\lambda,\alpha,\theta}$ has purely singular continuous spectrum \cite{gordon1976point,simon1982almost} if $|\lambda|>1$.
\end{remark}

Now we  will discuss  the histories  of Conjecture 1 and  also our approach to the proof of Theorem \ref{MaintheoremAnderson}.
We state another related  conjecture first.
Define
\begin{equation*}
    \delta(\alpha,\theta)=\limsup_{n\to \infty}\frac{-\ln ||2\theta+n\alpha||}{|n|}.
\end{equation*}
{\bf Conjecture 2:} Jitomirskaya  \cite{Conjecture} conjectured  that
\begin{description}
 \item[2a](Diophantine phase) $H_{\lambda,\alpha,\theta}$ satisfies Anderson localization if $|\lambda|>e^{\beta(\alpha)}$  and $\delta(\alpha,\theta)=0$, and $H_{\lambda,\alpha,\theta}$
  has purely singular continuous spectrum for all $\theta$ if  $1<|\lambda|<e^{\beta(\alpha)}$.
  \item[2b] (Diophantine frequency)   Suppose $\beta(\alpha)=0$.  $H_{\lambda,\alpha,\theta}$ satisfies Anderson localization if $|\lambda|>e^{\delta(\alpha,\theta)}$ and
  has purely singular continuous spectrum if  $1<|\lambda|<e^{\delta(\alpha,\theta)}$.
\end{description}

Notice that $\beta(\alpha)=0$ for almost every $\alpha$, and  $\delta(\alpha,\theta)=0$ for almost every $\theta$ and fixed $\alpha$.

The  case   $\beta(\alpha)=0$ and $\delta(\alpha,\theta)=0$ of Conjecture 2 was solved by  Jitomirskaya in her   pioneering  paper \cite{jitomirskaya1999metal}.
 Avila and Jitomirskaya   \cite{avila2009ten} proved  the localization part for Diophantine phases in the regime $|\lambda|>e^{\frac{16}{9}\beta}$, which was a key step to solve the  Ten Martin Problem. Liu and Yuan followed  their proof  and extended  the  result to $|\lambda|>e^{\frac{3}{2}\beta} $\cite{MR3340177}.
      Liu and Yuan  \cite{liu2015anderson} further developed      Avila-Jitomirskaya's    technics  in \cite{avila2009ten}   and verified  the Conjecture 1 in regime $|\lambda|>e^{7\beta}$. Here,  $\frac{3}{2}$ and  $7$ are the limit of the method of  \cite{avila2009ten}.

      Recently, Avila-You-Zhou \cite{AYZ} proved the singular continuous spectrum part of 2a, as well as  the measure-theoretic  version of 2a$:H_{\lambda,\alpha,\theta}$ satisfies Anderson localization for   $ |\lambda| >   e^{  \beta}$ and  almost every $\theta$. See also \cite{JK}.
       Diophantine frequency (2b) and  localization part of  Diophantine phase (2a)  were proved by Jitomirskaya and Liu \cite{jl,jl1}, who developed Avila-Jitomirskaya's scheme and found a better way to  deal with the phase and frequency resonances.

 One of the ideas of    \cite{jl,jl1} is that they  treat   the values of the  generalized eigenfunction at resonant points as variables and obtain the localization via solving the
  equations  of resonant points, not just using  block expansion and the exponential  decay of the Green functions. We should mention that the Green's functions are not necessarily exponential decay  in \cite{jl,jl1} and also in the present paper.

We want to explain  the motivations for Conjectures 1 and 2, and also explain  the new challenge for completely resonant phases.
For Diophantine frequency $\beta(\alpha)=0$, the resonant points come from the phase resonances\footnote{Roughly speaking, if $||2\theta+k\alpha||$ is small, $k$ is called a  phase resonance.  }.
   For Diophantine phase $\delta(\alpha,\theta)=0$, the resonant points come from the  frequency resonances\footnote{Roughly speaking, if $||k\alpha||$ is small, $k$ is called a frequency resonance.  }.
   Phase resonances lead to reflective repetitions of potential \cite{js} and frequency resonances  lead to  repetitions of potential \cite{gordon1976point,simon1982almost}.
Indeed, all known proofs of localization, for example \cite{fsw,bg,bgs,bourgain2002absolutely}, are based, in one
way or another, on avoiding
resonances and removing resonance-producing parameters.
For AMO and $|\lambda|>1$, the Lyapunov exponent is $\ln |\lambda|$.
Conjecture 2 says that the competition between  the Anderson localization and the singular continuous spectrum is actually the competition between the Lyapunov exponent and the strength of the resonance.
2a says that without phase resonances,  if  the Lyapunov exponent beats  the frequency resonance, then Anderson localization follows. Otherwise, $H_{\lambda,\alpha,\theta}$  has purely singular continuous spectrum.
2b says that without frequency resonances,  if the  Lyapunov exponent beats  the phase resonance, then Anderson localization follows. Otherwise, $H_{\lambda,\alpha,\theta}$  has purely singular continuous spectrum.

For   completely resonant phases $2\theta\in \alpha\mathbb{Z}+\mathbb{Z}$, $\delta(\alpha,\theta)=\beta(\alpha)$.
 Thus phase resonances and frequency resonances happen at the same time.  
 Conjecture 1 says that if   the Lyapunov exponent beats  the frequency  resonance  plus the  phase resonance, then the Anderson localization follows.
This  is the first challenge  in our paper since we need to deal with frequency  and phase resonances simultaneously.
The second challenge is to avoid the complete resonance.
In dealing with Conjecture 1, the original  arguments of Jitomirskaya \cite{jitomirskaya1999metal} do not work directly since there is the complete resonance.
In \cite{jitomirskaya2005localization}, Jitomirskaya-Koslover-Schulteis   found a trick to avoid the complete resonance by
  shrinking  the size of the interval around 0 (we refer it as ``shrinking scale'' technic). Later, the shrinking scale  technic was fully explored in \cite{liu2015anderson,MR3338640,han2017dry,avila2008almost}.
  It is a  natural idea to develop the shrinking scale technic and the localization arguments in  \cite{jl,jl1} to treat our situation.
 Since we shrink the scale,   there is one phase resonance {\it and} one
frequency phase resonance in a half scale. It is different from the situation in  Conjecture 2, where there is one phase resonance {\it or}  one frequency resonance in one scale.
Using full strength of   the localization proof of \cite{jl,jl1} to treat both  phase resonances and frequency resonances,
one  can only obtain the  Anderson localization for  $|\lambda|>e^{4\beta}$ in Conjecture 1, where  $4$ is the non-trivial technical limit in such approach.
  We bring several new ingredients that go beyond the technique of \cite{liu2015anderson,jitomirskaya2005localization,MR3338640,han2017dry,avila2008almost,jl,jl1} and allow us to improve the constant to 3, thus going well beyond the previous technical limit.
  In particular,
  instead of  using    Lagrange interpolation uniformly, we treat Lagrange interpolation individually   during the process of  finding the  points without  ``small divisors''.  This    gives  us  significantly more varieties to  construct  Green functions.
  We believe our method has a wider applicability to Anderson localization.

 \section{ Some notations and   known facts  }
It is well known that in order to
    prove  Anderson localization of   $H_{\lambda,\alpha,\theta}$,  we only need  to show  the following statements \cite{berezanskii1968expansions}:
     assume $\phi$ is a generalized   function, i.e.,
  \begin{equation*}
    H\phi=E\phi  ,\text{ and }  |\phi(k)|\leq  1+|k|, \text { for some }E,
  \end{equation*}
  then there exists some constant $c>0$ such that
  \begin{equation*}
    | \phi(k)|\leq C e^{-c|k|} \text{ for all } k.
  \end{equation*}

    It suffices to consider $ \alpha$ with  $ 0<\beta(\alpha) <\infty$.   Without loss of generality, we assume $\lambda>e^{3\beta}$, $\theta\in\{\frac{\alpha}{2},\frac{\alpha}{2}+\frac{1}{2},0, \frac{1}{2}\}$(shift is a unitary operator). In order to avoid too many notations,
 we still use  $2\theta\in \alpha\mathbb{Z}+\mathbb{Z}$  to represent  $\theta\in\{\frac{\alpha}{2},\frac{\alpha}{2}+\frac{1}{2},0, \frac{1}{2}\}$.
We also assume $E\in \Sigma_{\lambda,\alpha}$ (denote by  $\Sigma_{\lambda,\alpha}$ the spectrum of operator $H_{\lambda,\alpha,\theta}$ since the spectrum  does not depend on $\theta$).  For simplicity, we usually omit the dependence on parameters $E,\lambda,\alpha, \theta$.
  \par
  Given      a  generalized  eigenfunction $\phi$ of $H_{\lambda,\alpha,\theta}$, without loss of generality assume $\phi(0)=1$.
  Our objective is to show that there exists some specific $c>0$  such that
  $$| \phi(k)|\leq  e^{-c|k|} \text{ for } k\rightarrow \infty.$$
\par

Let us denote
$$ P_k(\theta)=\det(R_{[0,k-1]}(H_{\lambda,\alpha,\theta}-E) R_{[0,k-1]}).$$
It is easy to see that $P_k(\theta)$ is an even function of $ \theta+\frac{1}{2}(k-1)\alpha$  and can be written as a polynomial
      of degree $k$ in $\cos2\pi (\theta+\frac{1}{2}(k-1)\alpha )$ :

     \begin{equation}\label{GP_k}
      P_k(\theta)=\sum _{j=0}^{k}c_j\cos^j2\pi (\theta+\frac{1}{2}(k-1)\alpha)    \triangleq  Q_k(\cos2\pi  (\theta+\frac{1}{2}(k-1)\alpha)).
     \end{equation}
\begin{lemma} (\text{p}. 16, \cite{avila2009ten}) \label{Le.Numerater}
The following inequality  holds
    \begin{equation*}
     \lim_{k\rightarrow\infty}\sup_{\theta\in \mathbb{R}} \frac{1}{k} \ln | P_k(\theta)|\leq \ln \lambda.
    \end{equation*}
\end{lemma}

    \par
    By Cramer's rule (see p.15, \cite{bourgain2004green}for example)  for given  $x_1$ and $x_2=x_1+k-1$, with
     $ y\in I=[x_1,x_2] \subset \mathbb{Z}$,  one has
     \begin{eqnarray}
       |G_I(x_1,y)| &=&  \left| \frac{P_{x_2-y}(\theta+(y+1)\alpha)}{P_{k}(\theta+x_1\alpha)}\right|,\label{Cramer1}\\
       |G_I(y,x_2)| &=&\left|\frac{P_{y-x_1}(\theta+x_1\alpha)}{P_{k}(\theta+x_1\alpha)} \right|.\label{Cramer2}
     \end{eqnarray}
By Lemma \ref{Le.Numerater}, the numerators in  (\ref{Cramer1}) and (\ref{Cramer2}) can be bounded uniformly with respect to $\theta$. Namely,
for any $\varepsilon>0$,
\begin{equation}\label{Numerator}
    | P_n(\theta)|\leq e^{(\ln \lambda+\varepsilon)n}
\end{equation}
for large enough $n$.
\begin{definition}\label{Def.Regular}
Fix $t > 0$. A point $y\in\mathbb{Z}$ will be called $(t,k)$ regular if there exists an
interval $[x_1,x_2]$  containing $y$, where $x_2=x_1+k-1$, such that
\begin{equation*}
  | G_{[x_1,x_2]}(y,x_i)|\leq e^{-t|y-x_i|} \text{ and } |y-x_i|\geq \frac{1}{7} k \text{ for }i=1,2.
\end{equation*}
\end{definition}
It is  easy to check that (p. 61, \cite{bourgain2004green})
 \begin{equation}\label{Block}
   \phi(x)= -G_{[x_1 ,x_2]}(x_1,x ) \phi(x_1-1)-G_{[x_1 ,x_2]}(x,x_2) \phi(x_2+1),
 \end{equation}
 where  $ x\in I=[x_1,x_2] \subset \mathbb{Z}$.

      Given  a set $\{\theta_1, \cdots ,\theta_{k+1}\}$,  the lagrange Interpolation terms $La_i$, $i=1,2,\cdots,k+1$, are  defined by
      \begin{equation}\label{Def.Uniform}
       La_i= \ln \max_{ x\in[-1,1]} \prod_{ j=1 , j\neq i }^{k+1}\frac{|x-\cos2\pi\theta_j|}
        {|\cos2\pi\theta_i-\cos2\pi\theta_j|}.
      \end{equation}

        The following lemma is another form of  Lemma 9.3 in \cite{avila2009ten}.
      \begin{lemma}\label{Le.Uniform}
      Given a set $\{\theta_1, \cdots ,\theta_{k+1}\}$,   there exists some $\theta_i$ in set  $\{\theta_1, \cdots ,\theta_{k+1}\}$ such that

     \begin{equation*}
     P_{k}(\theta_i -\frac{k-1}{2}\alpha)\geq \frac{e^{k\ln\lambda-La_i}}{k+1}.
  \end{equation*}

      \end{lemma}
      \begin{proof}

     Otherwise,   for all $i=1,2,\cdots,k+1$,

     \begin{equation*}
        Q_{k}(\cos2\pi \theta_i) =P_{k}(\theta_i -\frac{k-1}{2}\alpha)<\frac{e^{k\ln\lambda-La_i}}{k+1}.
  \end{equation*}
 By (\ref{GP_k}),  we can write the polynomial $Q_k(x)$ in the Lagrange
     interpolation form at points $ \cos2\pi  \theta_i$, $i=1,2,\cdots,k+1$. Thus
     \begin{eqnarray*}
         |Q_k(x)| &=& \left| \sum_{i=1}^{k+1} Q_k(\cos2\pi\theta_i)\frac{\prod_{j\neq i }(x-\cos2\pi\theta_{j})}{\prod_{j\neq i}(\cos2\pi\theta_i-\cos2\pi\theta_{j})}\right| \\
        &<&  (k+1)\frac{e^{k\ln\lambda-La_i}}{k+1}e^{La_i}=e^{k\ln\lambda}
     \end{eqnarray*}
     for all $x\in[-1,1]$. By (\ref{GP_k}) again, $|P_k(x)|< e^{k\ln \lambda}$ for  all $ x\in \mathbb{R}$.
    However,  by Herman's subharmonic function methods (see  p.16   \cite{bourgain2004green}), $ \int_{\mathbb{R}/ \mathbb{Z}} \ln|P_{k}(x)|dx\geq k\ln \lambda$.  This is impossible.
      \end{proof}

     Fix a sufficiently small constant $\eta $, which will be determined later. Let $b_n= \eta q_{n}$.
     For any  $y\neq 0$,
         we will distinguish between two cases:
         \par
        (i)   $\text{dist}(y,  q_n\mathbb{Z}+\frac{q_n}{2}\mathbb{Z} )\leq b_n$,  called $n$-resonance.
          \par
        (ii)     $\text{dist}(y, q_n\mathbb{Z}+\frac{q_n}{2}\mathbb{Z} )> b_n$, called  $n$-nonresonance.
\par


\begin{theorem}(\cite{liu2015anderson})\label{Th.Nonresonant}
Assume $2\theta\in \alpha \mathbb{Z}+\mathbb{Z}$ and  $  \lambda>1$.

Suppose  either

 i)  $b_n\leq |y|< Cb_{n+1}$ for some $C>1$ and $y$ is $n$-nonresonant

  or

 ii)$  |y|\leq Cq_n$ and ${\rm dist}(y, q_n\mathbb{Z}+\frac{q_n}{2}\mathbb{Z} )> b_n$.

Let  $n_0$ be the least positive integer such that $4q_{n-n_0}\leq   {\rm dist}(y,q_n\mathbb{Z}+\frac{q_n}{2}\mathbb{Z} ) -2$.  Let $s\in\mathbb{N} $  be the largest number   such that $4sq_{n-n_0} \leq {\rm dist}(y,q_n\mathbb{Z}+\frac{q_n}{2}\mathbb{Z} )-2$.
  Then for
any $\varepsilon>0$ and   sufficiently large $n$,
$ y$ is $  (\ln \lambda -\varepsilon,6sq_{n-n_0}-1)$ regular.
\end{theorem}
The proof of  Theorem \ref{Th.Nonresonant}  builds on the ideas used in the proof  of Lemma B.4 in \cite{jl}, which is original from \cite{avila2009ten}. However it requires  some modifications  to avoid the completely resonant phases. Thus we give the proof in the Appendix.

 The following lemma can  be proved directly  by block expansion and Theorem \ref{Th.Nonresonant}, which is similar to the proof of  Lemma 4.1 in  \cite{jl}.
 We also  give the proof in the Appendix.
\begin{lemma}\label{Le.resonant}
 Suppose $k\in [jq_n,(j+\frac{1}{2})q_n]$  or $k\in [(j+\frac{1}{2})q_n,(j+1)q_n]$  with $0\leq |j| \leq  C\frac{b_{n+1}}{q_n}+C$,  and  $  {\rm dist}(k, q_n\mathbb{Z}+\frac{q_n}{2}\mathbb{Z} )\geq     10\eta q_n$. Let  $d_t=|k-tq_n|$ for $t\in\{j,j+\frac{1}{2},j+1\}$.
Then for sufficiently large
$n $,
\begin{equation}\label{Intervalk}
    |\phi(k)|\leq\max\{ r_j \exp\{-(\ln \lambda- \eta)(d_j-3\eta q_n)\},r_{j+\frac{1}{2}}\exp\{-(\ln \lambda- \eta)(d_{j+\frac{1}{2}}-3\eta q_n)\}\},
\end{equation}
or
\begin{equation}\label{Intervalk1}
    |\phi(k)|\leq\max\{ r_{j+\frac{1}{2} }\exp\{-(\ln \lambda- \eta)(d_{j+\frac{1}{2}}-3\eta q_n)\},r_{j+1}\exp\{-(\ln \lambda- \eta)(d_{j+1}-3\eta q_n)\}\}.
\end{equation}

\end{lemma}
\section{ Proof of Theorem \ref{MaintheoremAnderson}}

We always assume $n$ is large enough and $C$ is a large constant below. Denote by $\lfloor x\rfloor$ the  largest integer less or equal  than  $x$.

  Let
  \begin{equation*}
  r_j= \sup_{|r|\leq 10\eta }|\phi(jq_n+rq_n)|,
 \end{equation*}
 and
 \begin{equation*}
   r_{j+\frac{1}{2}}= \sup_{|r|\leq 10\eta }|\phi(jq_n+\lfloor\frac{q_n}{2}\rfloor+rq_n)|.
 \end{equation*}

We prove a crucial theorem first.
\begin{theorem}\label{Th.resonant}
Let $   |\ell |\leq \frac{b_{n+1}}{q_n}+3$. Then except $r_0$, we have
\begin{equation}
      r_{\ell}  \leq   \exp\{- (\ln \lambda-3\beta -C\eta)|\ell |q_n\},
 \end{equation}
and
\begin{equation}
      r_{\ell-\frac{1}{2}}  \leq   \exp\{- (\ln \lambda-3\beta -C\eta)|\ell-\frac{1}{2} |q_n\}.
 \end{equation}
\end{theorem}

\begin{lemma}\label{Le.halfr_j}
For any $ |  j |\leq  4\frac{b_{n+1}}{q_n}+16 $, the following holds,
\begin{equation*}
    r_{j+\frac{1}{2}}\leq \exp\{-\frac{1}{2}(\ln \lambda-2\beta- C\eta)q_n\} \max\{ r_{j},r_{j+1}\}.
\end{equation*}

\end{lemma}
\begin{proof}
Take  $\phi(j q_n+\lfloor\frac{q_n}{2}\rfloor+rq_n)$ with $|r|\leq 10\eta $ into consideration.
Without loss of generality assume $j\geq0$.
Let $n_0$ be the least positive integer such that
\begin{equation*}
 \frac{1}{\eta}  q_{n-n_0}\leq    (\frac{1}{6}-2\eta) q_n.
\end{equation*}
  Let $s$ be the
largest positive integer such that $sq_{n-n_0}\leq (\frac{1}{6}-2\eta) q_n $. Then
\begin{equation*}
    s\geq \frac{1}{\eta} .
\end{equation*}
By the fact $(s+1)q_{n-n_0}\geq (\frac{1}{6}-2\eta)q_n $,
 one has
 \begin{equation}\label{sq}
  (\frac{1}{6}-3\eta)q_n \leq sq_{n-n_0}\leq (\frac{1}{6}-2\eta)q_n.
 \end{equation}

Set $I_1, I_2\subset \mathbb{Z}$ as follows
\begin{eqnarray*}
  I_1 &=& [-2sq_{n-n_0},-1], \\
   I_2 &=& [ j q_n+\lfloor\frac{q_n}{2}\rfloor -(s+\lfloor\eta s\rfloor)  q_{n-n_0},j q_n+\lfloor\frac{q_n}{2}\rfloor+(s+\lfloor\eta s\rfloor) q_{n-n_0}-1 ],
\end{eqnarray*}
and let $\theta_m=\theta+m\alpha$ for $m\in I_1\cup I_2$. The set $\{\theta_m\}_{m\in I_1\cup I_2}$
consists of $(4s+2\lfloor \eta s\rfloor ) q_{n-n_0}$ elements. Let $k=(4s+2\lfloor \eta s\rfloor ) q_{n-n_0}-1$.

By  modifying  the proof of \cite[Lemma 9.9]{avila2009ten} or   \cite[Lemma 4.1]{liu2015anderson}, we can prove the claim (Claim 1):
  for any $\varepsilon>0$,   $  m\in I_1$, one has $La_m\leq  \varepsilon q_n$; and for any
$  m\in I_2$, one has $La_m\leq  q_n (\beta+ \varepsilon) $.
We also give the proof in the Appendix.

By Lemma \ref{Le.Uniform}, there exists some $j_0\in I_1$ such that
  $P_k(\theta_{j_0}-\frac{k-1}{2}\alpha) \geq e^{k\ln\lambda -\varepsilon q_n}$, or some $j_0\in I_2$ such that $P_k(\theta_{j_0}-\frac{k-1}{2}\alpha) \geq e^{k\ln\lambda-(\beta +\varepsilon)q_n}$.
  \par
  Suppose $j_0\in I_1$, i.e.,  $P_k(\theta_{j_0}-\frac{k-1}{2}\alpha) \geq e^{k\ln\lambda -\varepsilon q_n}$.  Let  $I=[j_0-2sq_{n-n_0}-\lfloor s\eta \rfloor q_{n-n_0}+1,j_0+2sq_{n-n_0}+\lfloor s\eta\rfloor q_{n-n_0}-1]=[x_1,x_2]$.  Denote by   $x_1^\prime=x_1-1$ and $x_2^\prime=x_2+1$.

   By
    (\ref{Cramer1}),  (\ref{Cramer2}) and (\ref{Numerator}),
 it is easy to verify
 \begin{eqnarray*}
   |G_I(0,x_i)| &\leq& e^{(\ln\lambda+\varepsilon )(k-2-|x_i|)- k\ln\lambda+\varepsilon q_n } \\
     &\leq&   e^{-|x_i|\ln \lambda +C\varepsilon q_n}.
 \end{eqnarray*}

Using (\ref{Block}) and noticing that  $|x_i|\geq \frac{\eta s }{2}q_{n-n_0}$, we obtain
\begin{equation}\label{Iterationr_j}
    |\phi(0)|  \leq \sum_{i=1,2}  e^{-\frac{\eta s}{2} q_{n-n_0}\ln \lambda+ C\varepsilon q_n}|\phi(x_i^\prime)|< 1,
\end{equation}
where the second inequality holds by (\ref{sq}).
  This  is contradicted to the fact $ \phi(0)=1$.
  \par
  Thus there exists  $j_0\in I_2$ such that $P_k(\theta_{j_0}-\frac{k-1}{2}\alpha) \geq e^{k\ln\lambda-(\beta +\varepsilon)q_n}$.
      Let  $I=[j_0-2sq_{n-n_0}-\lfloor s\eta \rfloor q_{n-n_0}+1,j_0+2sq_{n-n_0}+\lfloor s\eta\rfloor q_{n-n_0}-1]=[x_1,x_2]$.   By   (\ref{Cramer1}), (\ref{Cramer2}) and (\ref{Numerator}) again,
 we   have
\begin{equation}\label{G.Green}
|G_I(p,x_i)|\leq e^{(\ln\lambda+\varepsilon )(k-2-|p-x_i|)-k\ln\lambda + \beta q_n +\varepsilon q_n},
\end{equation}
where $p=jq_n+\lfloor\frac{q_n}{2}\rfloor+rq_n$.
Using (\ref{Block}), we obtain
\begin{equation}\label{Iterationr_j}
    |\phi(p)|  \leq \sum_{i=1,2} e^{( \beta +C\eta)q_{n}}|\phi(x_i^\prime)|e^{-|p-x_i|\ln \lambda }.
\end{equation}
Let $d_{i,i_1,i_2}=  |x_i- i_1q_n- i_2\frac{q_n}{2}|  $, where $i=1,2$, $i_1\in \mathbb{Z}$ and $i_2=0,1$.
If $d_{i,i_1,i_2}\geq 10 \eta q_n$, then we replace $\phi(x_i)$  in  (\ref{Iterationr_j}) with (\ref{Intervalk}) (or (\ref{Intervalk1})).
If   $d_{i,i_1,i_2}\leq 10 \eta q_n$, then we replace $\phi(x_i^\prime)$  in  (\ref{Iterationr_j}) with  $r_{i_1+\frac{i_2}{2}}$.
Then
we have
\begin{eqnarray}
   r_{j+\frac{1}{2}} &\leq& \max\{ \exp\{-\frac{1}{2}(\ln \lambda-2\beta- C\eta)q_n\} r_{j}, \exp\{-\frac{1}{2}(\ln \lambda-2\beta- C\eta)q_n\}r_{j+1}, \nonumber\\
   & & \exp
     \{-2sq_{n-n_0}\ln\lambda+\beta q_n+C\eta q_n\}r_{j+\frac{1}{2}}\}. \label{GHalfr_j1}
\end{eqnarray}

 By (\ref{sq}),
 one has
 \begin{eqnarray*}
   -2sq_{n-n_0}\ln\lambda+\beta q_n+C\eta q_n &<& (-\frac{\ln\lambda}{3}+\beta+C\eta)q_n \\
  &<&  0,
 \end{eqnarray*}
 for small $\eta$.
 This implies
 \begin{equation*}
 r_{j+\frac{1}{2}}
 \leq    \exp
     \{-2(s\eta+s)q_{n-n_0}\ln\lambda+\beta q_n+C\eta q_n\}r_{j+\frac{1}{2}}
 \end{equation*}
 can not happen.
 \par
Thus (\ref{GHalfr_j1}) becomes
\begin{equation}\label{GHalfr_j2}
    r_{j+\frac{1}{2}}\leq \max\{ \exp\{-\frac{1}{2}(\ln \lambda-2\beta- C\eta)q_n\} r_{j}, \exp\{-\frac{1}{2}(\ln \lambda-2\beta- C\eta)q_n\}r_{j+1}\}.
\end{equation}

\end{proof}
\begin{lemma}\label{Le.r_j}
For $   1 \leq |j| \leq  4\frac{b_{n+1}}{q_n}+12$, the following holds
\begin{equation}\label{Gr_j}
     r_{j }\leq \max\{\max_{t\in O}\{ \exp\{-(   |t| \ln\lambda-\beta -C\eta)q_n \} r_{j+t}\},  \exp\{-(\ln\lambda-3 \beta -C\eta )q_n \}r_{\pm 1}\} ,
\end{equation}
where $O=\{\pm\frac{3}{2} ,\pm\frac{1}{2}\}$.

\end{lemma}
\begin{proof}


 It suffices to estimate $\phi(j q_n +rq_n)$ with $|j|\geq 1$  and $|r|\leq 10\eta  $. Without loss of generality assume $j\geq1$.
 Let $n_0$ be the least positive integer such that
\begin{equation*}
 \frac{1}{\eta} q_{n-n_0}\leq \frac{ q_n}{6} -2.
\end{equation*}
 Let $s$ be the
largest positive integer such that $sq_{n-n_0}\leq  \frac{q_n}{6}-2 $.
Then  $s\geq \frac{1}{\eta}$.

 \par
Set $J_1, J_2,J_3\subset \mathbb{Z}$ as follows
\begin{eqnarray*}
  J_1 &=& [-2sq_{n-n_0},-1], \\
   J_2 &=& [ j q_n-3sq_{n-n_0},j q_n-2sq_{n-n_0}-1 ]\cup [ j q_n+2sq_{n-n_0},j q_n+3sq_{n-n_0}-1 ],\\
   J_3 &=& [ j q_n-2sq_{n-n_0},j q_n+2sq_{n-n_0}-1 ],
\end{eqnarray*}
and let $\theta_m=\theta+m\alpha$ for $m\in J_1\cup J_2\cup J_3$. The set $\{\theta_m\}_{m\in J_1\cup J_2\cup J_3}$
consists of $8sq_{n-n_0}$ elements.
By modifying the proof of \cite[Lemma 9.9]{avila2009ten} or   \cite[Lemma 4.1]{liu2015anderson} again, we can prove the claim (Claim 2)
 that   for any $ m\in J_1\ \cup J_3$ and any $\varepsilon>0$, $La_m\leq 2(\beta+\varepsilon) q_n$,
 and for any $m\in J_2$, $La_m\leq (\beta+\varepsilon) q_n$. We also give the details of proof  in the Appendix.

   Applying   Lemma  \ref{Le.Uniform}, there exists some $j_0$ with  $j_0\in J_1\cup J_3$
    such that
    \begin{equation*}
        P_{8sq_{n-n_0}-1}(\theta_{j_0}- (4sq_{n-n_0}-1)\alpha)\geq e^{8sq_{n-n_0}\ln \lambda-2\beta q_n-\varepsilon q_n},
    \end{equation*}
    or
    there exists some $j_0$ with  $j_0\in J_2$
    such that
    \begin{equation*}
        P_{8sq_{n-n_0}-1}(\theta_{j_0}-(4sq_{n-n_0}-1)\alpha)\geq e^{8sq_{n-n_0}\ln \lambda-\beta q_n-\varepsilon q_n}.
    \end{equation*}
    If $j_0\in J_2$, let  $I=[j_0- 4sq_{n-n_0}+1 ,j_0+4sq_{n-n_0}-1]=[x_1,x_2]$, then
    \begin{equation}\label{G.Green}
|G_I(jq_n +rq_n,x_i)|\leq e^{(\ln\lambda+\eta )(8sq_{n-n_0}-2-| jq_n +rq_n-x_i|)-8sq_{n-n_0}\ln\lambda + \beta q_n+C\eta q_n}.
\end{equation}
Using (\ref{Block}), we obtain
\begin{equation}\label{Iterationr_j2}
    |\phi(j q_n+rq_n)|  \leq \sum_{i=1,2} e^{(\beta +C\eta) q_{n}}|\phi(x_i^\prime)|e^{-|jq_n+rq_n-x_i|\ln \lambda }.
\end{equation}

  Recall that  $d_{i,i_1,i_2}=  |x_i- i_1q_n- i_2\frac{q_n}{2}|  $, where $i=1,2$, $i_1\in \mathbb{Z}$ and $i_2=0,1$.
If $d_{i,i_1,i_2}\geq 10 \eta q_n$, then we replace $\phi(x_i^\prime)$  in  (\ref{Iterationr_j2}) with (\ref{Intervalk})(or (\ref{Intervalk1})).
If   $d_{i,i_1,i_2}\leq 10 \eta q_n$, then we replace $\phi(x_i^\prime)$  in (\ref{Iterationr_j2}) with   $r_{i_1+\frac{i_2}{2}}$.

Then by \eqref{Iterationr_j2},
we have
\begin{equation*}
     r_{j }\leq  \exp\{ \beta q_n+C\eta q_n\} \max\{\max_{t\in O}\{ \exp\{- |t|q_n \ln\lambda\} r_{j+t}, \exp\{ -2sq_{n-n_0}\ln\lambda\}r_j\}\},
\end{equation*}
where $O=\pm\frac{3}{2},\pm1,\pm\frac{1}{2}$.
\par
Noting $sq_{n-n_0}\geq (1-\eta)\frac{1}{6}q_n$ (using $(s+1)q_{n-n_0}>\frac{1}{6}q_n-2$ and $ s\geq\frac{1}{\eta}$),
then
\begin{equation*}
r_{j}\leq  \exp\{ \beta q_n+C\eta q_n\} \exp\{ -2sq_{n-n_0}\ln\lambda\}r_j
\end{equation*}
 can not happen since $\ln \lambda>3\beta$.

Thus
\begin{equation*}
     r_{j }\leq \max_{t\in O}\{ \exp\{ \beta q_n+C\eta q_n - |t|q_n \ln\lambda\} r_{j+t} \},
\end{equation*}
where $O=\pm\frac{3}{2},\pm1,\pm\frac{1}{2}$.
This implies (\ref{Gr_j}).
\par
If $j_0\in J_3$, by the same arguments, we have
\begin{equation*}
     r_{j }\leq \max_{t\in \{\pm1,\pm\frac{1}{2}\}}\{ \exp\{2 \beta q_n+C\eta q_n - |t|q_n \ln\lambda\} r_{j+t} \}.
\end{equation*}
Using the estimate of $r_{j\pm\frac{1}{2}}$ in  Lemma \ref{Le.halfr_j},
we have
\begin{equation*}
     r_{j }\leq \exp\{-(\ln\lambda-3\beta-C\eta)q_n\} \max\{ r_{j\pm1},r_{j} \}.
\end{equation*}
By the same reason,
\begin{equation*}
r_{j}\leq  \exp\{-(\ln\lambda-3\beta-C\eta)q_n\}r_{j}
\end{equation*}
 can not happen.
 Thus
 \begin{equation}\label{Gsecondadd}
     r_{j }\leq   \exp\{-(\ln\lambda-3\beta-C\eta)q_n\}r_{j\pm1}.
\end{equation}
This also implies (\ref{Gr_j}).
\par
If $j_0 \in J_1$, then (\ref{Gsecondadd}) holds for $j=0$, which will  lead to
$|\phi(0)|< 1$. This is impossible.
\end{proof}
\textbf{Proof of Theorem \ref{Th.resonant}}
\begin{proof}

 By Lemmas  \ref{Le.halfr_j} and \ref{Le.r_j},  for  $ 1\leq   j  \leq 2\frac{b_{n+1}}{q_n}+4 $, we have
 \begin{equation}\label{G.r_j2}
     r_{j-\frac{1}{2} }\leq  \exp\{ - \frac{1}{2}(\ln \lambda-3\beta-C\eta) q_n \} \max\{r_{j-1},r_{j}\} ,
\end{equation}
and
\begin{equation}\label{G.r_j1}
     r_{j }\leq  \max_{t\in O}\{\exp\{ -|t|(\ln \lambda-3\beta-C\eta) q_n \} r_{j+t}\} ,
\end{equation}
where $O=\{\pm\frac{3}{2},\pm1,\pm\frac{1}{2}$\}.
 For  $ -\frac{b_{n+1}}{q_n}-3\leq   j \leq -1 $, we have
 \begin{equation}\label{G.r_j2oct}
     r_{j+\frac{1}{2} }\leq  \exp\{ - \frac{1}{2}(\ln \lambda-3\beta-C\eta) q_n \} \max\{r_{j+1},r_{j}\} ,
\end{equation}
and
\begin{equation}\label{G.r_j1oct}
     r_{j }\leq  \max_{t\in O}\{\exp\{ -|t|(\ln \lambda-3\beta-C\eta) q_n \} r_{j+t}\} .
\end{equation}


\par
Suppose $\ell>0$.
Let $j=\ell$ in   (\ref{G.r_j1}) and (\ref{G.r_j2}),   and iterate     $ 2\ell$  times or until $j\leq 1$,
 we obtain
 \begin{equation}\label{radd1}
     r_{\ell} \leq  (2\ell+2)q_n\exp\{- (\ln \lambda-3\beta -C\eta) \ell q_n\},
 \end{equation}
 and
 \begin{equation}\label{radd2}
 r_{\ell-\frac{1}{2}}\leq  (2\ell+2)q_n\exp\{- (\ln \lambda-3\beta -C\eta) (\ell -\frac{1}{2})q_n\}.
 \end{equation}
Notice that  we have used the fact that $|r_{j}|\leq (|j|+2)q_n$ and $|r_{j-\frac{1}{2}}|\leq (|j-\frac{1}{2}|+2)q_n$.

Suppose $\ell<0$.
Let $j=\ell$ in   (\ref{G.r_j1oct}) and (\ref{G.r_j2oct}),   and iterate    $2|\ell|$ times  or until $j\geq -1$,
 we obtain
 \begin{equation}\label{radd1oct}
     r_{\ell} \leq  (2\ell+2)q_n\exp\{- (\ln \lambda-3\beta -C\eta) |\ell| q_n\},
 \end{equation}
 and
 \begin{equation}\label{radd2oct}
 r_{\ell+\frac{1}{2}}\leq  (2\ell+2)q_n\exp\{- (\ln \lambda-3\beta -C\eta) |\ell +\frac{1}{2}|q_n\}.
 \end{equation}
Now   Theorem \ref{Th.resonant} follows from (\ref{radd1}),  (\ref{radd2}), (\ref{radd1oct}) and   (\ref{radd2oct}).
 \end{proof}
 {\bf Proof of  Theorem \ref{MaintheoremAnderson}}
\begin{proof}
Without loss of generality, we assume $k>0$.
Let $\eta>0$ be   much smaller than $\ln\lambda-3\beta$.
For any    $k$,
let $n$ be such that  $b_n\leq k<b_{n+1}$.

Case 1: $  {\rm dist}(k, q_n\mathbb{Z}+\frac{q_n}{2}\mathbb{Z} )\leq     10\eta q_n$.

In this case, applying Theorem \ref{Th.resonant}, one has
\begin{equation}\label{Decayoct1}
      | \phi(k)| ,  | \phi(k-1)|\leq   \exp\{- (\ln \lambda-3\beta -C\eta)|k|\}.
\end{equation}

Case 2: $  { \rm dist}(k, q_n\mathbb{Z}+\frac{q_n}{2}\mathbb{Z} )\geq     10\eta q_n$.

Let $0\leq j\leq \frac{b_{n+1}}{q_n}$ such that $k\in [jq_n,(j+\frac{1}{2})q_n]$  or $k\in [(j+\frac{1}{2})q_n,(j+1)q_n]$.

By Lemma \ref{Le.resonant} and
Theorem \ref{Th.resonant}, one also has
\begin{equation}\label{Decayoct2}
      | \phi(k)|, | \phi(k-1)|\leq   \exp\{- (\ln \lambda-3\beta -C\eta)|k|\}.
\end{equation}
By \eqref{Decayoct1},\eqref{Decayoct2} and letting $\eta\to 0$,
we have
\begin{equation*}
    \limsup_{k\to \infty} \frac{\ln (\phi^2(k)+\phi^2(k-1))}{2|k|} \leq -(\ln \lambda-3\beta).
\end{equation*}
We finish the proof.

\end{proof}

\appendix
\section{Proof of Theorem \ref{Th.Nonresonant}, Claims 1 and  2}

 Let $ \frac{p_n}{q_n}$ be the continued fraction approximations to $\alpha$, then
\begin{equation}\label{GDC1}
\forall 1\leq k <q_{n+1},  \text{dist}( k\alpha,\mathbb{Z})\geq  |q_n\alpha -p_n|,
\end{equation}
and
\begin{equation}\label{GDC2}
      \frac{1}{2q_{n+1}}\leq|q_n\alpha-p_n| \leq\frac{1}{q_{n+1}}.
\end{equation}
\begin{lemma} (\text{Lemma } 9.7, \cite{avila2009ten})
Let $\alpha\in \mathbb{R}\backslash \mathbb{Q}$, $x\in\mathbb{R}$ and $0\leq \ell_0 \leq q_n-1$ be such that
$ |\sin\pi(x+\ell_0\alpha)|=\inf_{0\leq\ell\leq q_n-1}    |\sin\pi(x+\ell \alpha)|$, then for some absolute constant $C > 0$,
\begin{equation}\label{G927}
    -C\ln q_n\leq \sum _{\ell=0\atop \ell\neq \ell_0}^{q_n-1} \ln|\sin\pi(x+\ell\alpha )|+(q_n-1)\ln2\leq  C\ln q_n.
\end{equation}
  \end{lemma}

  {  \bf Proof of Theorem \ref{Th.Nonresonant}}
\begin{proof}
We only give the proof of  case 1: $b_n\leq |y|< Cb_{n+1}$ is non-resonant.

By the definition of $s$ and $n_0$, we have
$4sq_{n-n_0}\leq {\rm dist}(y,q_n \mathbb{Z})-2$ and  $4q_{n-n_0+1}> {\rm dist}(y,q_n \mathbb{Z})-2$. This leads to $sq_{n-n_0}\leq q_{n-n_0+1}$.
Set $I_1, I_2\subset \mathbb{Z}$ as follows
\begin{eqnarray*}
  I_1 &=& [-2sq_{n-n_0},-1], \\
   I_2 &=& [ y-2sq_{n-n_0},y+2sq_{n-n_0}-1 ],
\end{eqnarray*}
and let $\theta_j=\theta+j\alpha$ for $j\in I_1\cup I_2$. The set $\{\theta_j\}_{j\in I_1\cup I_2}$
consists of $6sq_{n-n_0}$ elements.
\par
Let $k=6sq_{n-n_0}-1$.
 We estimate $ La_i$ first. For this reason,
 let $x=\cos2\pi a$,   and take    the logarithm in (\ref{Def.Uniform}), one has
  \begin{equation*}
   \ln \prod_{ j\in I_1\cup  I_2  \atop j\neq i } \frac{|\cos2\pi a-\cos2\pi\theta_j|}
        {|\cos2\pi\theta_i-\cos2\pi\theta_j|}\;\;\;\;\;\;\;\;\;\;\;\;\;\;\;\;\;\;\;\;\;\;\;\;\;\;\;\;\;\;\;\;\;\;\;\;\;\;\;\;\;\;\;\;\;\;\;\;\;\;\;\;\;
        \;\;\;\;\;\;\;\;\;\;\;\;\;\;\;
  \end{equation*}
         \begin{equation*}
          =   \sum _{ j\in I_1\cup  I_2  \atop j\neq i }\ln|\cos2\pi a-\cos2\pi \theta_j|- \sum _{ j\in I_1\cup  I_2  \atop j\neq i }\ln|\cos2\pi\theta_i  -\cos2\pi \theta_j| .
         \end{equation*}

We  start to estimate $ \sum _{ j\in I_1\cup  I_2  , j\neq i }\ln|\cos2\pi a-\cos2\pi \theta_j| $.
Obviously,
   $$ \sum _{ j\in I_1\cup  I_2 \atop j\neq i }\ln|\cos2\pi a-\cos2\pi \theta_j| \;\;\;\;\;\;\;\;\;\;\;\;\;\;\;\;\;\;\;\;\;\;\;\;\;\;\;\;\;\;\;\;\;\;\;\;\;\;\;\;\;\;\;\;\;\;\;\;\;\;\;\;\;\;\;\;$$
$$\;\;\;\;\;\;\;\;\;\;\;\;\;\;\;\;\;\;\;\;\;=\sum_{ j\in I_1\cup  I_2  \atop j\neq i }\ln|\sin\pi(a+\theta_j)|+\sum_{ j\in I_1\cup  I_2  \atop j\neq i }\ln |\sin\pi(a-\theta_j)|
+(6sq_{n-n_0}-1)\ln2  $$
\begin{equation*}
    =\Sigma_{+}+\Sigma_-+(6sq_{n-n_0}-1)\ln2.  \;\;\;\;\;\;\;\;\;\;\;\;\;\;\;\;\;\;\;\;\; \;\;\;\;\;\;\;\;\;\;\;\;\;\;\;\;\;\;\;\;\;\;\;\;\;\;\;\;\;\;\;\;
\end{equation*}
Both $\Sigma_+$ and $\Sigma_-$ consist of $6s$ terms of the form of  (\ref{G927}), plus 6s terms of the form
\begin{equation*}
    \ln\min_{j=0,1,\cdots,q_{n-n_0}}|\sin\pi(x+j\alpha)|,
\end{equation*}
minus $\ln|\sin\pi(a\pm\theta_i)|$.
       Thus,  using   (\ref{G927})  6s times of $\Sigma_{+}$ and $\Sigma_{-}$ respectively,  one has
 \begin{equation}\label{G.appnumerator}
   \sum_{j \in I_1  \cup I_2\atop j\neq i}\ln|\cos2\pi a-\cos2\pi \theta_{j}|\leq-6sq_{n-n_0}\ln2+Cs\ln q_{n-n_0}.
\end{equation}
Let $a=\theta_i$,
we obtain
   $$ \sum _{j \in I_1  \cup I_2\atop j\neq i}\ln|\cos2\pi \theta_i-\cos2\pi \theta_j| \;\;\;\;\;\;\;\;\;\;\;\;\;\;\;\;\;\;\;\;\;\;\;\;\;\;\;\;\;\;\;\;\;\;\;\;\;\;\;\;\;\;\;\;\;\;\;\;\;\;\;\;\;\;\;\;$$
$$\;\;\;\;\;\;\;\;\;\;\;\;\;\;\;\;\;\;\;\;\;=\sum_{j \in I_1  \cup I_2\atop j\neq i}\ln|\sin\pi(\theta_i+\theta_j)|+\sum_{j \in I_1  \cup I_2\atop j\neq i}\ln |\sin\pi(\theta_i-\theta_j)|
+(6sq_{n-n_0}-1)\ln2  $$
\begin{equation}\label{G.appsumdenumerate}
    =\Sigma_{+}+\Sigma_-+(6sq_{n-n_0}-1)\ln2,  \;\;\;\;\;\;\;\;\;\;\;\;\;\;\;\;\;\;\;\;\;\;\;\; \;\;\;\;\;\;\;\;\;\;\;\;\;\;\;\;\;\;\;\;\;\;\;\;\;\;\;\;\;\;\;
\end{equation}
 where
  \begin{equation*}
   \Sigma_{+}=\sum_{j \in I_1  \cup I_2\atop j\neq i}\ln |\sin\pi(2\theta+ (i+j) \alpha)|,
 \end{equation*}
 and
\begin{equation*}
     \Sigma_-=\sum_{j \in I_1  \cup I_2 \atop j\neq i}\ln |\sin\pi( i-j)\alpha|.
 \end{equation*}
 We will estimate $\Sigma_+ $.
 Set
$J_1=[-2s,-1]$ and
$J_2=[0 ,4s-1 ]$, which are two adjacent disjoint intervals of length
$2s$ and
$4s$ respectively.
 Then $I_1\cup
I_2$ can be represented as a disjoint union of segments $B_j,\;j\in
J_1\cup J_2,$ each of length $q_{n-n_0}$.

Applying (\ref{G927}) to  each  $B_j$, we
obtain
\begin{equation}\label{G312}
\Sigma_+ \geq -6sq_{n-n_0}\ln 2+
\sum_{j\in J_1\cup J_2 }\ln
|\sin  \pi\hat \theta_j|-Cs\ln q_{n-n_0}-\ln|\sin2\pi (\theta+i\alpha)|,
\end{equation}
where
\begin{equation}\label{G313}
|\sin  \pi\hat \theta_j|=\min_{\ell \in B_j}|\sin  \pi
(2\theta +(\ell +  i)\alpha )|.
\end{equation}

 By the  construction of $I_1$ and $I_2$, one has
\begin{equation}\label{G314}
   2\theta +(\ell +  i)\alpha=\pm (mq_n\alpha+r_1\alpha)\;\mod \mathbb{Z}
\end{equation}
or
\begin{equation}\label{G315}
   2\theta +(\ell +  i)\alpha= \pm r_2\alpha \mod \mathbb{Z},
\end{equation}
where $0\leq m\leq C \frac{b_{n+1}}{q_n}$ and $1\leq r_i<q_n$, $i=1,2$.

By (\ref{GDC1}) and  (\ref{GDC2}), it follows
\begin{eqnarray}
\nonumber
  \min_{\ell \in I_1\cup I_2}\ln| \sin \pi(2\theta +(\ell +  i)\alpha)|&\geq&  C\ln(||r_i\alpha||_{\mathbb{R}/\mathbb{Z}}-\frac{\Delta_{n-1}}{2} ) \\
  \nonumber
    &\geq& C\ln( \Delta_{n-1}-\frac{\Delta_{n-1}}{2}) \\
    &\geq&\ln  C\frac{\Delta_{n-1}}{2}\geq   -C\ln q_n,\label{G316new}
\end{eqnarray}
since $ || mq_n\alpha||_{\mathbb{R}/\mathbb{Z}}\leq C\frac{\eta q_{n+1}  }{q_n} \Delta_n\leq\frac{\Delta_{n-1}}{2}$.

By the construction of $I_1$ and $I_2$,
we also have
\begin{equation}\label{Gsum-}
     \min_{i\neq j\atop i,j\in I_1\cup I_2} \ln |\sin
 \pi (j-i )\alpha) | \geq -C\ln q_n.
\end{equation}

Next we  estimate $\sum_{j\in J_1  }\ln
|\sin  \pi\hat \theta_j|$.
Assume that $\hat \theta_{j+1}=\hat \theta_j+ q_{n-n_0}  \alpha$ for
every $j,j+1 \in J_1$.
In this case, for any $i,j\in J_1$ and $i\neq j$, we have
\begin{equation}\label{APPaddnew1}
    ||\hat\theta_i-\hat\theta_j||_{\mathbb{R}/\mathbb{Z}}\geq  ||q_{n-n_0}\alpha||_{\mathbb{R}/\mathbb{Z}}.
\end{equation}

By   the Stirling formula,  (\ref{G316new}) and (\ref{APPaddnew1}),  one has
 \begin{eqnarray}
  \nonumber
   \sum_{j\in J_1}\ln
|\sin 2\pi\hat\theta_j|&>&   2\sum_{j=1}^{s}\ln
( j\Delta_{n-n_0}) -C\ln q_n\\
    &>&  2s\ln\frac s{q_{n-n_0+1}}-C\ln q_n-Cs.\label{G317}
 \end{eqnarray}
 \par
In the other cases, decompose $J_1$ into maximal intervals $T_\kappa$
such that for $j,j+1 \in
T_\kappa$ we have $\hat \theta_{j+1}=\hat \theta_j+ q_{n-n_0}
\alpha$.  Notice that the boundary points of an interval $T_\kappa$
are either boundary points of $J_1$ or satisfy
$\| \hat \theta_j\|_{\mathbb{R}/\mathbb{Z}}+\Delta_{n-n_0} \geq
\frac {\Delta_{n-n_0-1}} {2}$.
This follows from the fact that if
$0<|z|<q_{n-n_0}$,  then $\|  \hat \theta_j+q_{n-n_0}\alpha\|_{\mathbb{R}/\mathbb{Z}}\leq \|  \hat \theta_j\|_{\mathbb{R}/\mathbb{Z}}+\Delta_{n-n_0}$,  and  $\|  \hat \theta_j+(z+q_{n-n_0})\alpha \|_{\mathbb{R}/\mathbb{Z}}\geq   \|z \alpha\|_{\mathbb{R}/ \mathbb{Z}}-\|  \hat \theta_j+q_{n-n_0}\alpha\|_{\mathbb{R}/\mathbb{Z}}
\geq \Delta_{n-n_0-1}-\|  \hat \theta_j\|_{\mathbb{R}/\mathbb{Z}}-\Delta_{n-n_0}$.
Assuming $T_\kappa \neq J_1$, then there exists $j \in
T_\kappa$ such that $\|  \hat \theta_j\|_{\mathbb{R}/\mathbb{Z}}\geq
\frac {\Delta_{n-n_0-1}} {2}- \Delta_{n-n_0} $.
\par
If $T_\kappa$  contains  some $j$ with $ \|  \hat
\theta_j\|_{\mathbb{R}/\mathbb{Z}}<\frac {\Delta_{n-n_0-1}} {10} $, then
 \begin{eqnarray}
 \nonumber
   |T_\kappa| &\geq & \frac{\frac {\Delta_{n-n_0-1}} {2}-\Delta_{n-n_0} -\frac {\Delta_{n-n_0-1}} {10}}{\Delta_{n-n_0}}  \\
    &\geq&\frac{1}{4} \frac{\Delta_{n-n_0-1}}{ \Delta_{n-n_0}}-1\geq \frac{s}{8}-1,\label{G318}
 \end{eqnarray}
 since $sq_{n-n_0}\leq q_{n-n_0+1}$,  where $|T_\kappa|=b-a+1$ for $T_\kappa=[a,b]$.
For such $ T_\kappa$,
 a similar estimate to (\ref{G317}) gives
\begin{eqnarray}
\nonumber
  \sum_{j\in T_\kappa}\ln
|\sin  \pi\hat\theta_j|  &\geq&   |T_\kappa|\ln \frac {|T_\kappa|}{q_{n-n_0+1}} -C s-C\ln q_n  \\
 &\geq&    |T_\kappa|\ln \frac{s}{q_{n-n_0+1}}-Cs-C\ln q_n. \label{G319}
\end{eqnarray}
 If $T_\kappa$ does not   contain any $j$ with $ \|  \hat
\theta_j\|_{\mathbb{R}/\mathbb{Z}}<\frac {\Delta_{n-n_0-1}} {10} $, then by (\ref{GDC2})
\begin{eqnarray}
\nonumber
  \sum_{j\in T_\kappa}\ln
|\sin  \pi\hat\theta_j|  &\geq&  -|T_\kappa|\ln q_{n-n_0}-C|T_\kappa| \\
    &\geq&    |T_\kappa|\ln \frac{s}{q_{n-n_0+1}}-C|T_\kappa|.\label{G320}
\end{eqnarray}
 \par
By (\ref{G319}) and (\ref{G320}), one has
\begin{equation}\label{G321}
  \sum_{j\in J_1}\ln
|\sin  \pi\hat\theta_j|\geq  2s\ln \frac{s}{q_{n-n_0+1}} -C  s -C\ln q_n.
\end{equation}
Similarly,
\begin{equation}\label{G322}
  \sum_{j\in J_2}\ln
|\sin  \pi\hat\theta_j|\geq  4s\ln \frac{s}{q_{n-n_0+1}} -C  s-C\ln q_n .
\end{equation}
Putting  (\ref{G312}), (\ref{G321}) and (\ref{G322}) together, we have
\begin{equation}\label{G323}
\Sigma_+ > -6sq_{n-n_0}\ln 2+ 6s\ln \frac{s}{q_{n-n_0+1}} -Cs\ln q_{n-n_0}-C\ln q_n.
\end{equation}
Now we start to estimate $ \Sigma_-$.

Replacing (\ref{G316new}) with (\ref{Gsum-}), and following the proof of (\ref{G323}), we obtain,
\begin{equation}\label{newG323}
\Sigma_- > -6sq_{n-n_0}\ln 2+ 6s\ln \frac{s}{q_{n-n_0+1}} -Cs\ln q_{n-n_0}-C\ln q_n.
\end{equation}

By  (\ref{G.appsumdenumerate}), (\ref{G323}) and (\ref{newG323}), we obtain
$$ \sum _{j \in I_1  \cup I_2\atop j\neq i}\ln|\cos2\pi \theta_i-\cos2\pi \theta_j| \;\;\;\;\;\;\;\;\;\;\;\;\;\;\;\;\;\;\;\;\;\;\;\;\;\;\;\;\;\;\;\;\;\;\;\;\;\;\;\;\;\;\;\;\;\;\;\;\;\;\;\;\;\;\;\;$$
\begin{equation}\label{G.app325}
  \geq -6sq_{n-n_0}\ln 2+ 6s\ln \frac{s}{q_{n-n_0+1}} -Cs\ln q_{n-n_0}-C\ln q_n.
\end{equation}
By (\ref{G.appnumerator})  and (\ref{G.app325}),
we have for any $ i\in I_1\cup I_2$,
\begin{equation*}
         \prod_{j \in I_1\cup I_2\atop j \neq i } \frac{|x-\cos2\pi\theta_{j }|}
        {|\cos2\pi\theta_i-\cos2\pi\theta_{j }|}\leq e^{ 6sq_{n-n_0}(-2\ln (s /q_{n-n_0+1})/q_{n-n_0}+  \varepsilon  )  }.
      \end{equation*}
Using the fact  $4(s+1)q_{n-n_0}>\eta q_{n} -2$, one has
for any $ i\in I_1\cup I_2$,
\begin{equation}\label{G43}
         \prod_{j \in I_1\cup I_2 \atop j \neq i } \frac{|x-\cos2\pi\theta_{j }|}
        {|\cos2\pi\theta_i-\cos2\pi\theta_{j }|}\leq e^{ sq_{n-n_0} \varepsilon    }.
\end{equation}
This implies $La_i\leq \varepsilon sq_{n-n_0}  $ for any $i=1,2,\cdots k+1$, where $k =6sq_{n-n_0}-1$.

    Applying  Lemma  \ref{Le.Uniform}, there exists some $j_0$ with  $j_0\in I_1\cup I_2$ such that
    \begin{equation*}
        P_{k-1}(\theta_{j_0}-\frac{k-1}{2}\alpha)\geq e^{ (\ln \lambda-\varepsilon)k }.
    \end{equation*}

      Firstly, we assume $j_0\in I_2$.
      \par
      Set $I=[j_0-3sq_{n-n_0}+1,j_0+3sq_{n-n_0}-1]=[x_1,x_2]$.   By  (\ref{Cramer1}), (\ref{Cramer2}) and (\ref{Numerator}) again,
 one has
\begin{equation*}
|G_I(y,x_i)|\leq\exp\{(\ln\lambda+\varepsilon  )(6sq_{n-n_0}-1-|y-x_i|)-6sq_{n-n_0}(\ln\lambda-  \varepsilon) \}.
\end{equation*}
Notice that $|y-x_i|\geq sq_{n-n_0}$,
we obtain
\begin{equation}\label{2AppG.Green}
|G_I(y,x_i)|\leq\exp\{ -(\ln \lambda-  \varepsilon)|y-x_i| \}.
\end{equation}
\par
If $j_0\in I_1$,  we may let $y=0$ in (\ref{2AppG.Green}).
By  (\ref{Block}), we get
\begin{equation*}
    |\phi(0)| \leq 6sq_{n-n_0}\exp\{ -(\ln \lambda-  \varepsilon)sq_{n-n_0}\}.
\end{equation*}
This   contradicts $\phi(0)=1$. Thus $j_0\in I_2$, and the theorem follows from (\ref{2AppG.Green}).
\end{proof}
{\bf Proof of Claim 1}
\begin{proof}
By the construction of $I_1$ and $I_2$ in Claim 1,   \eqref{GDC1} and \eqref{GDC2},
we have  for $i \in I_1$,
\begin{equation}\label{Appenoct1}
     \min_{\ell \in I_1\cup I_2}\ln |\sin \pi(2\theta +(\ell +  i)\alpha)|\geq -C\ln q_n,
\end{equation}
and
\begin{equation}\label{Appenoct2}
     \min_{i\neq j\atop j\in I_1\cup I_2} \ln |\sin
 \pi (j-i )\alpha) | \geq -C\ln q_n.
\end{equation}

Replacing \eqref{G316new} with \eqref{Appenoct1} and \eqref{Gsum-} with \eqref{Appenoct2}, and following the proof of \eqref{G43},
we can show that  for any $i\in I_1$,
\begin{equation*}
         \prod_{j \in I_1\cup I_2\atop j \neq i } \frac{|x-\cos2\pi\theta_{j }|}
        {|\cos2\pi\theta_i-\cos2\pi\theta_{j }|}\leq e^{   \varepsilon  sq_{n-n_0}  }.
      \end{equation*}
      This implies  for $i\in I_1$, $La_i\leq \varepsilon q_n$.

      By the construction of $I_1$ and $I_2$ in Claim 1,   \eqref{GDC1} and  \eqref{GDC2} again,
we have  for $i \in I_2$,
\begin{equation}\label{Appenoct3}
     \min_{\ell \in I_1\cup I_2}\ln |\sin \pi(2\theta +(\ell +  i)\alpha)|_{\mathbb{R}/\mathbb{Z}}\geq  -\beta q_n-C\ln q_n,
\end{equation}
and
\begin{equation}\label{Appenoct4}
     \min_{i\neq j\atop j\in I_1\cup I_2} \ln |\sin
 \pi (j-i )\alpha) | \geq -C\ln q_n.
\end{equation}
We should mention that, for each $i\in I_2$, there is exact one $ j\in I_1\cup I_2$ such that the lower bound of \eqref{Appenoct3} can be achieved.

Replacing \eqref{G316new} with \eqref{Appenoct3} and \eqref{Gsum-} with \eqref{Appenoct4}, and following the proof of \eqref{G43},
we can show that  for any $i\in I_1$,
\begin{equation*}
         \prod_{j \in I_1\cup I_2 \atop j \neq i } \frac{|x-\cos2\pi\theta_{j }|}
        {|\cos2\pi\theta_i-\cos2\pi\theta_{j }|}\leq e^{   \varepsilon  sq_{n-n_0} +\beta q_n }.
      \end{equation*}
      This implies for any
$  i\in I_2$, $La_i\leq  q_n (\beta+ \varepsilon) $.

\end{proof}
{\bf Proof of Claim 2}

\begin{proof}
Let $J_3^1=[ j q_n-2sq_{n-n_0},j q_n-1]$ and $J_3^2=[jq_n,+2sq_{n-n_0}-1] $ so that $J_3=J_3^1\cup J_3^2$.
Let $I=J_1\cup J_2\cup J_2$.

{\bf Case 1:} $i \in J_1\cup J_3^1$

By the construction of $J_1$, $J_2$  and $J_3$ in Claim 2, and \eqref{GDC1}, \eqref{GDC2},  we have
\begin{equation}\label{Appenoct5}
     \min_{\ell \in I}\ln |\sin \pi(2\theta +(\ell +  i)\alpha)|\geq -\beta q_n-C\ln q_n,
\end{equation}
and
\begin{equation}\label{Appenoct6}
     \min_{i\neq j\atop j\in I } \ln |\sin
 \pi (j-i )\alpha) | \geq -\beta q_n-C\ln q_n.
\end{equation}
Moreover, there are exact two $\ell,j \in I$ such that the lower bound of \eqref{Appenoct5}  can be achieved  for $\ell$ and the lower bound of \eqref{Appenoct6} can be achieved for $j$.

{\bf Case 2:} $i \in J_1\cup J_3^2$

By the  same reason, we have
\begin{equation}\label{Appenoct7}
     \min_{\ell \in I}\ln |\sin\pi(2\theta +(\ell +  i)\alpha)|\geq -\beta q_n-C\ln q_n,
\end{equation}
and
\begin{equation}\label{Appenoct8}
     \min_{i\neq j\atop j\in I } \ln |\sin
 \pi (j-i )\alpha) | \geq -C\ln q_n.
\end{equation}
Moreover, there are exact two $\ell_1,\ell_2 \in I$ such that the lower bound of \eqref{Appenoct7}  can be achieved  for both $\ell_1$ and $\ell_2$.

{\bf Case 3:} $i \in J_2$

By the  same reason, we have
\begin{equation}\label{Appenoct9}
     \min_{\ell \in I}\ln |\sin\pi(2\theta +(\ell +  i)\alpha)|\geq -\beta q_n-C\ln q_n,
\end{equation}
and
\begin{equation}\label{Appenoct10}
     \min_{i\neq j\atop j\in I } \ln |\sin
 \pi (j-i )\alpha) | \geq -C\ln q_n.
\end{equation}
Moreover, there is exact one $\ell \in I$ such that the lower bound of \eqref{Appenoct9}  can be achieved  for $\ell$.

 Now following the proof of the Claim 1, we can prove Claim 2.

\end{proof}
\section{Proof of Lemma \ref{Le.resonant}}
Without loss of generality, we assume $k\in[jq_n,(j+\frac{1}{2})q_n]$ and $j\geq0$.
Let $d_j=k-jq_n$ and $d_{j+\frac{1}{2}}=(j+\frac{1}{2})q_n-k$.

 For any $y$ $  \in [j q_n+\eta q_{n },(j+\frac{1}{2}) q_n-
\eta q_{n }]$,   by  Theorem \ref{Th.Nonresonant},
  $y$ is regular with $\tau=\ln\lambda -\eta$.
Therefore
 there exists an interval $ I(y)=[x_1,x_2]\subset
[  j  q_n,(j+\frac{1}{2})q_n]$
such that $y\in I(y)$ and
\begin{equation}\label{G329}
    \text{dist}(y,\partial I(y))\geq  \frac{1}{7} |I(y)|   \geq  \frac{q_{n-n_0}}{2}
\end{equation}
and
\begin{equation}\label{G330}
  |G_{I(y)}(y,x_i)| \leq e^{-(\ln \lambda-  \eta)|y-x_i|},\;i=1,2,
\end{equation}
where
 $ \partial I(y)$ is the boundary of the interval $I(y)$, i.e.,$\{x_1,x_2\}$, and  $ |I(y)|$ is the  size of $I(y)\cap\mathbb{Z} $, i.e., $ |I(y)|=x_2-x_1+1$.
   For $z  \in  \partial I(y)$,  let
  $z' $ be the neighbor of $z$, (i.e., $|z-z'|=1$) not belonging to $I(y)$.
\par
If $x_2+1\leq (j+\frac{1}{2})q_n-\eta q_n$ or  $x_1-1\geq  j q_n+ \eta q_n$,
we can expand $\phi(x_2+1)$ or $\phi(x_1-1)$ using (\ref{Block}). We can continue this process until we arrive to $z$
such that $z+1>(j+\frac{1}{2})q_n- \eta q_n$ or  $z-1< j  q_n+ \eta q_n$, or the iterating number reaches
$\lfloor\frac{4 q_n}{ q_{n-n_0}}\rfloor$. Thus, by (\ref{Block})
\begin{equation}\label{G331}
   \phi(k)=\displaystyle\sum_{s ; z_{i+1}\in\partial I(z_i^\prime)}
G_{I(k)}(k,z_1) G_{I(z_1^\prime)}
(z_1^\prime,z_2)\cdots G_{I(z_s^\prime)}
(z_s^\prime,z_{s+1})\phi(z_{s+1}^\prime),
\end{equation}
where in each term of the summation one has
$j q_n+\eta q_{n }+1\leq z_i\leq (j+\frac{1}{2}) q_n-\eta
q_{n }-1$, $i=1,\cdots,s,$ and
  either $z_{s+1} \notin [j q_n+\eta q_{n }+1,(j+\frac{1}{2}) q_n-
 \eta q_{n }-1]$, $s+1 < \lfloor\frac{4 q_n}{ q_{n-n_0}}\rfloor$; or
$s+1= \lfloor\frac{4 q_n}{ q_{n-n_0}}\rfloor$.
We should mention that $z_{s+1}\in[jq_n,(j+\frac{1}{2})q_n] $.
\par
 If $z_{s+1} \in [j q_n,j q_n+ \eta q_{n }]$, $s+1 < \lfloor\frac{4 q_n}{ q_{n-n_0}}\rfloor$,
this implies
\begin{equation*}
    |\phi(z_{s+1}^\prime)|\leq r_j.
\end{equation*}
By  (\ref{G330}), we have
\begin{equation*}
    \nonumber
   | G_{I(k)}(k,z_1) G_{I(z_1^\prime)}
(z_1^\prime,z_2)\cdots G_{I(z_s^\prime)}
(z_s^\prime,z_{s+1})\phi(z_{s+1}^\prime)|
\end{equation*}
\begin{eqnarray}
\nonumber
&\leq & r_j=e^{-(\ln\lambda-\eta)(|k-z_1|+\sum_{i=1}^{s}|z_i^\prime-z_{i+1}|)}
 \\
\nonumber
&\leq & r_j e^{-(\ln\lambda-\eta)(|k-z_{s+1}|-(s+1))}  \\
&\leq & r_j  e^{-(\ln\lambda- \eta)(d_j- 2\eta q_{n } -4-\frac{ 4q_n}{ q_{n-n_0}})}.
\label{G332}
\end{eqnarray}
 If $z_{s+1} \in [(j+\frac{1}{2}) q_n-\eta q_n,(j+\frac{1}{2}) q_n ]$, $s+1 < \lfloor\frac{4 q_n}{ q_{n-n_0}}\rfloor$,
 by the same arguments,
  we have
  \begin{equation}\label{G.addGreen}
    | G_{I(k)}(k,z_1) G_{I(z_1^\prime)}
(z_1^\prime,z_2)\cdots G_{I(z_s^\prime)}
(z_s^\prime,z_{s+1})\phi(z_{s+1}^\prime)|\leq  r_{j+\frac{1}{2}} e^{-(\ln\lambda- \eta)(d_{j+\frac{1}{2}}- 2\eta q_{n } -4-\frac{ 4q_n}{ q_{n-n_0}})}.
  \end{equation}
  If $s+1= \lfloor\frac{4 q_n}{ q_{n-n_0}}\rfloor,$
using   (\ref{G329}) and (\ref{G330}), we obtain
\begin{equation}\label{G333}
     | G_{I(k)}(k,z_1) G_{I(z_1^\prime)}
(z_1^\prime,z_2)\cdots G_{I(z_s^\prime)}
(z_s^\prime,z_{s+1})\phi(z_{s+1}^\prime)|\leq e^{-(\ln\lambda-\eta) {\frac{1}{2}q_{n-n_0}} \lfloor\frac{4 q_n}{ q_{n-n_0}}\rfloor}|\phi(z_{s+1}^\prime)| .
\end{equation}
Notice that the total number of terms in (\ref{G331})
is  at most  $2^{\lfloor\frac{4 q_n}{ q_{n-n_0}}\rfloor}$ and $d_j,d_{j+\frac{1}{2}}\geq 10\eta q_n$.  By (\ref{G332}), (\ref{G.addGreen}) and (\ref{G333}),  we have
\begin{equation}\label{G.add1}
|\phi(k)|\leq  \max\{r_j e^{-(\ln\lambda-\eta) (d_j-3\eta q_n) }, r_{j+\frac{1}{2}} e^{-(\ln\lambda-\eta) (d_{j+\frac{1}{2}}-3\eta q_n) }, e^{-(\ln\lambda-\eta) q_n }\max_{p\in[jq_n,(j+\frac{1}{2})q_n]}|\phi(p)|\}.
\end{equation}
Now we will show that for any $p\in[jq_n,(j+\frac{1}{2})q_n]$, one has
$ |\phi(p)|\leq  \max\{ r_j,r_{j+\frac{1}{2}}\}$. Then (\ref{G.add1}) implies   Lemma \ref{Le.resonant}.
Otherwise, by the definition of $r_j$, if   $|\phi(p^\prime)|$ is the largest one of $|\varphi(z)|,z\in [jq_n+10\eta q_n+1,(j+\frac{1}{2})q_n-10\eta q_n-1]$,
then $|\phi(p^\prime)|>\max\{ r_j,r_{j+\frac{1}{2}}\}$. Applying (\ref{G.add1}) to $\phi(p^\prime) $ and noticing  that  ${\rm dist} (p^\prime,q_n\mathbb{Z})\geq 10\eta q_n$,
we get
\begin{equation*}
|\phi(p^\prime)|\leq   e^{-7(\ln\lambda-\eta)\eta q_n  } \max\{ r_j,r_{j+\frac{1}{2}},|\phi(p^\prime)|\}.
\end{equation*}
This is impossible  because $|\phi(p^\prime)|>\max\{ r_j,r_{j+\frac{1}{2}}\}$.
 \section*{Acknowledgments}
  I would like to thank Svetlana Jitomirskaya for comments on earlier versions of the manuscript.
   This research was  supported by  the AMS-Simons Travel Grant 2016-2018, NSF DMS-1401204 and NSF DMS-1700314.
The author is  grateful to the Isaac
    Newton Institute for Mathematical Sciences, Cambridge, for its
    hospitality, supported by EPSRC Grant Number EP/K032208/1, during the programme Periodic and Ergodic Spectral
    Problems where this work was started.

\end{document}